\documentclass[10pt,a4paper]{amsart}
\usepackage{hyperref}
\usepackage{mathrsfs}
\usepackage{amsmath}
\usepackage{amssymb}
\usepackage[all]{xy}
\usepackage{latexsym}
\usepackage[T1]{fontenc}
\title[Quadratic self-duality]{On quadratic coalgebras, duality and the universal Steenrod algebra}
\author[Geoffrey Powell]{Geoffrey M.L. Powell}
\address{Laboratoire Analyse, Géométrie et Applications, UMR 7539\\ Institut Galilée, Université Paris 13, 93430 Villetaneuse, France}
\email{powell@math.univ-paris13.fr} 
\thanks{This work was partly financed by the project ANR BLAN08-2 338236, HGRT}
\newtheorem{thm}{Theorem}[subsection]
\newtheorem{prop}[thm]{Proposition}
\newtheorem{cor}[thm]{Corollary}
\newtheorem{lem}[thm]{Lemma}

\theoremstyle{definition}
\newtheorem{defn}[thm]{Definition}
\newtheorem{exam}[thm]{Example}
\theoremstyle{remark}
\newtheorem{rem}[thm]{Remark}
\newtheorem{nota}[thm]{Notation}
\newtheorem{hyp}[thm]{Hypothesis}

\newcommand{\catmonoid}{\mathrm{Monoid}}
\newcommand{\grcatmon}{\mathrm{Monoid}^{\mathrm{gr}}}
\newcommand{\dualst}{\mathscr{A}^*}
\newcommand{\univsteen}{\mathscr{Q}}
\newcommand{\atilde}{\widetilde{\mathscr{A}}}
\newcommand{\monoid}{{\mathscr{M}}}
\newcommand{\bialg}{{\mathscr{B}}}

\newcommand{\grcoalg}{\mathrm{Coalg}}
\newcommand{\conngrcoalg}{\grcoalg^{\mathrm{conn}}}
\newcommand{\quadgrcoalg}{\mathrm{Coalg}^q}
\newcommand{\ftquadgrcoalg}{\mathrm{Coalg}^q_{ft}}
\newcommand{\ftquadgralg}{\mathrm{Alg}^q_{ft}}
\newcommand{\vs}[1][\kfield]{\mathscr{E}_{#1}}
\newcommand{\ctensor}{{T_c}}

\renewcommand{\phi}{\varphi}
\renewcommand{\epsilon}{\varepsilon}
\newcommand{\nat}{\mathbb{N}}
\newcommand{\zed}{\mathbb{Z}}
\newcommand{\cx}{\mathbb{C}}
\newcommand{\field}{\mathbb{F}_2}
\newcommand{\kfield}{\mathbb{K}}

\newcommand{\op}{^\mathrm{op}}
\newcommand{\cali}{\mathcal{I}}
\newcommand{\calj}{\mathcal{J}}
\newcommand{\calo}{\mathcal{O}}
\newcommand{\cals}{\mathcal{S}}
\newcommand{\calx}{\mathcal{X}}
\newcommand{\caly}{\mathcal{Y}}
\begin{document}
\begin{abstract}
The notion of quadratic self-duality for coalgebras is developed with applications to algebraic structures which arise naturally
in algebraic topology, related to the universal Steenrod algebra via an appropriate form of duality. This explains and unifies results of Lomonaco and Singer.
\end{abstract}
\maketitle

\section{Introduction}

In his work related to the cohomology of the Steenrod algebra, Singer \cite{singer_invt_lambda} introduced a graded coalgebra $\Gamma$ over the field $\field$ and  showed how to obtain the dual of the opposite of the Lambda
algebra as a quotient of $\Gamma$. In related work \cite{singer_hopf_ops}, Singer analysed the structure of a certain sub-coalgebra,
which is dual to the Steenrod algebra of cohomology operations for Hopf algebra cohomology. These results can be compared with
those of
Lomonaco on the universal Steenrod algebra (aka. the big Steenrod algebra introduced by May in his general algebraic approach to the construction of Steenrod operations \cite{may_gen_alg}), who showed  that the invariant theoretic
constructions of Singer lead to
a natural presentation of the universal Steenrod algebra as a quadratic algebra \cite{lomonaco_dickson}. There
is a phenomenon of quadratic self-duality which relates these results; this is implicit in earlier results in the literature, where it appears under the guise of
Koszul duality, for example in the results of Lomonaco on reciprocity \cite{lomonaco_reciprocity}; the relation with
Singer's work \cite{singer_hopf_ops}
has been investigated in \cite{lomonaco_coalgebra}.

This note develops the general theory of quadratic self-duality and illustrates the theory by giving an independent construction of Singer's coalgebra $\Gamma$ in the spirit of Milnor's description of the dual Steenrod algebra. As a biproduct of this approach, the Steenrod algebra appears by a process of group completion. A direct proof that the coalgebra is quadratic is given.

These results are then applied to the universal Steenrod algebra, which is identified as the associated quadratic algebra. An important fact is that, although the coalgebra $\Gamma$ is not of finite type, it is obtained by a localization of a quadratic coalgebra $\bialg^+$, which is of finite type. An analogue of this result for the universal Steenrod algebra can be given.

The paper has two parts, the first considers generalities on quadratic self-duality and the second applies these results in the context of the universal Steenrod algebra. Section \ref{sect:coalgebras} reviews material on quadratic coalgebras, introducing the notion of
admissibility (dual to the existence of monomial bases for quadratic algebras), a weak dual form of a PBW basis and the  concept of transpose duality for admissible quadratic coalgebras. Quadratic self-duality is introduced in Section \ref{sect:coalg_dual} both via vector space duality and using transpose duality. This is
applied in Section \ref{subsect:quadratic_push_pull} to prove the associated reciprocity results between pushforward and pullback quadratic coalgebras.

The model for $\Gamma$ is constructed in Section \ref{sect:bialgebras}, using a monoid-valued functor on the category of commutative $\field$-algebras.
The main result of the section is Theorem \ref{thm:bialg_quad}, which gives a direct proof that $\Gamma$ is a quadratic coalgebra.
The fundamental properties of $\Gamma$ are established in Section \ref{sect:duality}, where it is shown that $\Gamma$ is admissible and
that $\Gamma$ is quadratically self-dual (Theorem \ref{thm:bialg_selfdual}).

Finally, in Section \ref{sect:universal}, the connection with the universal Steenrod algebra is explained. The main result is Theorem
\ref{thm:univsteen_self-dual}, which shows that the universal Steenrod algebra is the quadratic algebra associated to $\Gamma$ and, hence,
is quadratically self-dual. 
\tableofcontents
\part{Quadratic self-duality}
\section{Quadratic coalgebras}
\label{sect:coalgebras}

\subsection{One-cogenerated and quadratic coalgebras}

For the convenience of the reader, the basic notions concerning graded coalgebras over a field $\kfield$ (see \cite{posit_vishik}, \cite{milnor_moore}).

The category of $\nat$-graded vector spaces is
symmetric monoidal with respect to the graded tensor product, with  unit $\kfield$ concentrated in degree zero.

\begin{defn}
	An  $\nat$-graded coalgebra is a counital comonoid in the category of $\nat$-graded
$\kfield$-vector spaces. A coalgebra $C$ is connected if the counit $\epsilon$ induces an isomorphism $C_0 \cong \kfield$. A morphism of $\nat$-graded coalgebras $C \rightarrow C'$ is a morphism of counital comonoids.
\end{defn}

 \begin{nota}
 Let $\vs$ denote the category of $\kfield$-vector spaces and $\grcoalg$ the category of $\nat$-graded coalgebras and
$\conngrcoalg \subset \grcoalg$ the full subcategory of connected coalgebras.

Let $\ctensor$ denote the tensor coalgebra functor $\ctensor : \vs \rightarrow \conngrcoalg$ which is right adjoint to the
functor $\conngrcoalg \rightarrow \vs$, which sends a graded coalgebra $(C, \Delta, \epsilon)$ to $C_1$.
 \end{nota}

\begin{defn}
\
\begin{enumerate}
	\item
	A graded connected coalgebra $(C, \Delta , \epsilon)$ is one-cogenerated if  the canonical morphism  $C \rightarrow \ctensor (C_1)$ is a monomorphism of graded coalgebras.
\item
	For $V$ a $\kfield$-vector space and a subspace $R \leq V^{\otimes 2}$, let $\langle V; R \rangle$ denote the
one-cogenerated graded coalgebra defined by
	\[
		\langle V; R \rangle_n := \left \{
		\begin{array}{ll}
		\kfield	&n=0
	\\
	V & n=1 \\
	\bigcap _{i+j+2 = n} V^{\otimes i } \otimes R \otimes V^{\otimes j} & n \geq 2.
		\end{array}
\right.
	\]
\item
A graded coalgebra $C$ is quadratic if there exists a pair $V, R\leq V^{\otimes 2}$  and an isomorphism of graded
coalgebras $C \cong \langle V; R \rangle$.
\item
	Let $\quadgrcoalg \subset \conngrcoalg$ denote the full subcategory of quadratic coalgebras.
\end{enumerate}
\end{defn}

\begin{prop}
\label{prop:quadratic_coalgebras}
The category $\quadgrcoalg$ is equivalent to the category of pairs of $\kfield$-vector spaces $(V, R \leq V
^{\otimes 2} ) $, where a morphism $(V, R \leq V ^{\otimes 2} ) \rightarrow (W, T \leq W^{\otimes 2} )$ is a morphism of
vector spaces $f : V \rightarrow W$ such that $f^{\otimes 2} $ restricts to a morphism $R \rightarrow T$.
\end{prop}

\begin{nota}
	For a pair of vector spaces $(V, R \leq V^{\otimes 2} )$, let $\{ V; R \}$ denote the associated quadratic
algebra, defined as $T (V) / \langle R \rangle$, where $T (V)$ is the tensor algebra on $V$ and $\langle R \rangle$ is the two-sided graded ideal generated by $R$.
 \end{nota}

\begin{prop}
\cite{posit_vishik}
	The category $\quadgrcoalg$ of quadratic coalgebras is equivalent to the category of quadratic algebras via the correspondence $\langle V ; R \rangle \leftrightarrow \{V ; R \}$.
\end{prop}

\subsection{Pullback and pushout constructions for  quadratic coalgebras}
\label{subsect:pushforward_pullback}

Fix a quadratic coalgebra $\langle V; R \rangle$.

\begin{defn}
	For $W \leq V$ a sub-vector space, let $\langle W, R_W\rangle $ denote the quadratic coalgebra defined by $R_W
:= R \cap W^{\otimes 2}$.
\end{defn}

\begin{prop}
	Let $i : W \hookrightarrow V$ be an inclusion of vector spaces, then there is a unique monomorphism of quadratic
coalgebras $\langle W,
R_W \rangle \hookrightarrow \langle V; R \rangle $ which makes the following diagram commute:
	\[
		\xymatrix{
\langle W ; R_W \rangle
\ar@{^(->}[r]
\ar@{^(->}[d]
&
\ctensor (W)
\ar@{^(->}[d]^{\ctensor (i)}
\\
\langle V ; R \rangle
\ar@{^(->}[r]
&
\ctensor (V).
		}
	\]
\end{prop}

\begin{rem}
	The underlying map  $\langle W; R_W \rangle \hookrightarrow \langle V; R \rangle$ is a  monomorphism of $\nat$-graded $\kfield$-vector spaces.
\end{rem}

\begin{defn}
	Let $V  \twoheadrightarrow Z $ be an epimorphism of vector spaces. The quadratic coalgebra $\langle Z; R^Z
\rangle$ is the quadratic coalgebra defined by the image $R^Z$ of $R$ in $Z^{\otimes 2}$.
\end{defn}

\begin{prop}
	There is a unique morphism of coalgebras $\langle V; R \rangle \rightarrow \langle Z; R^Z \rangle$ which makes
the following diagram commute
	\[
		\xymatrix{
\langle V; R \rangle
\ar@{^(->}[r]
\ar[d]
&
\ctensor (V)
\ar[d]
\\
\langle Z; R^Z \rangle
\ar@{^(->}[r]
&
\ctensor (Z).
		}
	\]
	Moreover, the morphism $\langle V; R \rangle \rightarrow \langle Z; R^Z \rangle$ is an epimorphism of quadratic coalgebras.
\end{prop}

\begin{rem}
\label{rem:epimorphism_not_surjection}
	In general  $\langle V; R \rangle \rightarrow \langle Z; R^Z \rangle$ is not a surjection of $\nat$-graded  $\kfield$-vector spaces.
\end{rem}

\subsection{The quadratic dual coalgebra}

To give the usual definition of the quadratic dual coalgebra, a finiteness hypothesis is required.

\begin{hyp}
\label{hyp:ft}
	Let $\langle V ; R \rangle $ be a quadratic coalgebra. Suppose that	$V$ is a $\zed$-graded $\kfield$-vector
space, $R \leq V^{\otimes 2}$ is a graded (homogeneous) subspace and $V \otimes V$ is a vector space of finite type.
\end{hyp}

\begin{nota}
\
\begin{enumerate}
 \item
Let $V^*$ denote the graded dual of the  graded $\kfield$-vector space $V$.
\item
Let $\ftquadgrcoalg$ denote the full subcategory of quadratic coalgebras satisfying Hypothesis \ref{hyp:ft} and
let $\ftquadgralg$ denote the analogous category of quadratic algebras of finite type.
\end{enumerate}
\end{nota}

\begin{lem}
Graded vector space duality induces an equivalence of categories
\[
	\Big( \ftquadgrcoalg \Big)
	\op
	\stackrel{\cong}{\rightarrow}
	\ftquadgralg.
\]
\end{lem}

\begin{defn}
	Let $C:= \langle V ; R \rangle $ be a quadratic coalgebra satisfying hypothesis \ref{hyp:ft}. The quadratic dual
$C^!$ of $C$ is the quadratic coalgebra $\{ V^* , R^\perp \}$, where $R^\perp $ is the kernel of $V^* \otimes V^*
\twoheadrightarrow R^*$.
\end{defn}

\begin{prop}
\label{prop:duality_quadratic}
	Let $C:= \langle V; R \rangle $ be an object of $\ftquadgrcoalg$.
	\begin{enumerate}
		\item
		The quadratic coalgebra $C^!$ is the dual to the quadratic algebra $\{ V; R \}$ (that is $\langle V; R
\rangle ^! \cong \{ V ; R \} ^*$).
		\item
		There is a natural isomorphism of quadratic coalgebras $C \cong (C^{!})^!$.
		\item
		Quadratic duality induces an equivalence of categories
		\[
			(\ftquadgrcoalg)\op \stackrel{\cong}{\rightarrow} \ftquadgrcoalg.
		\]
	\end{enumerate}
\end{prop}

\subsection{Admissibility}

\begin{nota}
	Let $\kfield [\cali ]$ denote the free $\kfield$-vector space on the set $\cali$. The canonical generator of $\kfield [\cali]$ associated to the element $i
\in \cali$ is denoted by $[i]$.
\end{nota}

To specify an isomorphism of $\kfield$-vector spaces $\kfield [\cali ] \stackrel{\cong}{\rightarrow} V$ is equivalent to
giving a choice of basis indexed by $\cali$.

\begin{lem}
\label{lem:proj-inj}
	Let $\calx, \caly$ be sets.
	\begin{enumerate}
\item
There are  canonical isomorphisms $\kfield [\calx \times \caly] \cong \kfield
[\calx] \otimes \kfield [\caly]$ and $\kfield [\calx \amalg \caly] \cong \kfield [\calx] \oplus \kfield
[\caly]$ of vector spaces.
\item
		There is a canonical monomorphism of $\kfield$-vector spaces  $\kfield [\calx ] \hookrightarrow \kfield
[\calx \amalg \caly]$  and a canonical split epimorphism $\kfield [\calx \amalg \caly ]
\twoheadrightarrow \kfield [\calx]$ given by
		\[
		[w \in \calx \amalg \caly]
		\mapsto
		\left\{
		\begin{array}{ll}
			0 & w \in \caly \\
			{[}w\in  \cals{]}& w \in \calx .
		\end{array}
\right.
		\]
	\end{enumerate}
\end{lem}

\begin{nota}
	Write $\cals'$ for the complement of a subset $\cals \subset \cali \times \cali $.
\end{nota}

\begin{defn}
	Let $(\cali , \cals \subset \cali \times \cali ) $ be a pair of sets. The quadratic coalgebra $\langle V; R
\rangle $ is $(\cali, \cals)$-admissible if there exists an isomorphism $\kfield [\cali] \stackrel{\cong}{\rightarrow}
V$ such that the composite $R\rightarrow \kfield [\cals]$ defined by the diagram
	\[
		\xymatrix{
		R \ar@{^(->}[r]
		\ar[d]_\cong
		&
		V \otimes V
			\ar[d]^\cong
				\\
		\kfield [\cals]
&
\kfield [\cali \times \cali ]
		\ar@{->>}[l]
		}
	\]
is an isomorphism, where the isomorphism $V \otimes V \cong \kfield [\cali \times \cali]$ is induced by the canonical
isomorphism  $\kfield[\cali \times \cali ] \cong \kfield [\cali] \otimes \kfield [\cali]$ and the given isomorphism
$\kfield [\cali ] \cong V$.
\end{defn}

\begin{rem}
	Compare \cite[Chapter 4, Section 1]{polis_posit} for conditions which establish the existence of admissible
bases (in the case of quadratic algebras).
\end{rem}

\begin{prop}
	Let $(\cali , \cals ) $ be as above.  The set of  $(\cali, \cals)$-admissible quadratic coalgebras correspond bijectively to the set of  functions
	\[
		f: \cals \times \cals' \rightarrow \kfield
	\]
such that, for each $s \in S$, the restriction $f (s, - ) : \cals' \rightarrow \kfield$ has finite support.
\end{prop}

\begin{proof}
Suppose that  $\langle \kfield [\cali] ; R \rangle $ is an  $(\cali, \cals)$-admissible quadratic coalgebra.
By hypothesis, there is an  isomorphism $R \cong \kfield [\cals ]$ with respect to which the inclusion $R \hookrightarrow
\kfield[\cali ] \otimes \kfield [\cali] \cong \kfield [\cali \times \cali] $ is defined by
	\begin{eqnarray}
	\label{eqn:basic_admissible}
		[s] \mapsto [s] - \Sigma_{s' \in \cals '} f (s, s') [s'],
	\end{eqnarray}
	for some coefficients $f (s, s') \in \kfield$, such that  the function $f (s,-) : s' \mapsto f(s,s')$ has finite support.

	Conversely, given such a function $f: \cals \times \cals ' \rightarrow \kfield$, the equation
(\ref{eqn:basic_admissible}) defines an inclusion $\kfield [\cals ] \hookrightarrow \kfield [\cali ] \otimes \kfield
[\cali]$ which gives an $(\cali, \cals)$-admissible quadratic coalgebra.
\end{proof}

\begin{nota}
	Denote the $(\cali, \cals)$-admissible quadratic coalgebra $\langle \kfield [\cali] ; R \rangle $ associated to
the
function $f : \cals \times \cals' \rightarrow \kfield$  by $\langle \cali;
\cals, f \rangle $.
\end{nota}

\subsection{Transpose duality}

For the purposes of this paper, a generalization of the quadratic duality functor is required, which allows the
finite-type hypothesis to be relaxed.

\begin{defn}
	The $(\cali, \cals)$-admissible quadratic coalgebra $\langle \cali; \cals, f \rangle $ is dualizable if the
function $f (-, s') : \cals \rightarrow \kfield$ has finite support, for every $s' \in \cals '$.
\end{defn}

\begin{nota}
	For $f: \cals \times \cals' \rightarrow \kfield$ a function, write $f^! : \cals ' \times \cals \rightarrow
\kfield$  for the function
	$f( s', s) = f (s,s')$, $\forall (s', s) \in \cals ' \times \cals$.
	\end{nota}

\begin{defn}
The transpose dual of the  dualizable $(\cali, \cals)$-admissible quadratic coalgebra $\langle \cali; \cals, f \rangle $
	is the quadratic coalgebra
	 $
		\langle \cali ; \cals , f \rangle^! := \langle \cali ; \cals' , f^! \rangle.
	 $
\end{defn}

\begin{rem}
	The terminology transpose duality is used to avoid confusion with that of quadratic duality. The two notions can
be related when $\cali$ is equipped with a $\zed$-grading (that is a function $\cali \rightarrow \zed$) and  $\kfield
[\cali \times \cali ]$ is of finite type with respect to this grading. In this case, the dual basis provides a canonical
isomorphism between
	$
		\kfield [\cali \times \cali ] ^*$ and $ \kfield [\cali \times \cali ]
	$
which relates the two notions.
\end{rem}

\subsection{Admissibility,  pullbacks and pushforwards}
\label{subsect:admissible_push_pull}

In general, the property of admissibility is not preserved under pullbacks or pushforwards; the following result gives an explicit criterion when taking a sub-basis of cogenerators.

\begin{nota}
For  $(\cali, \cals \subset \cali \times \cali) $ sets and $\calj
\subset \cali$ a subset, write $\cals_\calj:=\cals \cap \calj^{\times 2}$.
\end{nota}

\begin{prop}
\label{prop:admissible_pullback}
	Let $\langle \kfield [\cali]; R \rangle $ be an $(\cali, \cals)$-admissible quadratic coalgebra. For $\calj
\subset \cali$ a subset, the pullback quadratic coalgebra $\langle \kfield [\calj] ; R_{\kfield[\calj]} \rangle $ is
$(\calj, \cals_\calj)$-admissible if and only if the function $f : \cals \times \cals' \rightarrow \kfield$ associated
to $\langle \kfield [\cali]; R \rangle $ satisfies
\[
	f (s_\calj, s') =0
\]
for all $s_\calj \in \cals _\calj $ and $s' \in \cals' \backslash \cals'_\calj $.
\end{prop}

\begin{proof}
	Under the hypothesis on $\langle \kfield [\cali]; R \rangle$, there is a commutative diagram:
\[
	\xymatrix{
R_{\kfield[\calj]}
\ar@{^(->}[r]
\ar@{^(->}[d]
&
\kfield [\calj \times \calj]
\ar@{->>}[r]
\ar@{^(->}[d]
&
\kfield [\cals_\calj]
\ar@{^(->}[d]
\\
R
\ar[r]
\ar@/_1pc/[rr]_\cong
&
\kfield [\cali \times \cali]
\ar[r]
&
\kfield [\cals],
}
\]
where the left hand square is a pullback (by definition) and the right hand square is defined using the canonical
inclusions and projections of Lemma \ref{lem:proj-inj}; commutativity is a consequence of the definition of
$\cals_\calj$.

It follows that the composite of the top row is a monomorphism; commutativity of the left hand square shows that it is an isomorphism if and only if the given condition is satisfied.
\end{proof}

Similarly, for the pushforward of an admissible quadratic coalgebra:

\begin{prop}
\label{prop:admissible_pushforward}
	Let $\langle \kfield[\cali]; R \rangle $ be an $(\cali, \cals)$-admissible quadratic coalgebra. For $\calj
\subset \cali$ a subset,   the
pushforward quadratic coalgebra $\langle \kfield [\calj] ; R^{\kfield[\calj]} \rangle $ is $(\calj,
\cals_\calj)$-admissible if and only if the function $f : \cals \times \cals' \rightarrow \kfield$ associated to
$\langle \kfield [\cali]; R \rangle $ satisfies
\[
	f (s, s') =0
\]
for all $s  \in  \cals \backslash \cals _\calj $ and $s' \in  \cals'_\calj$.
\end{prop}

\begin{proof}
	The proof is formally dual to that of Proposition \ref{prop:admissible_pullback}.
\end{proof}

These conditions are related by transpose duality.

\begin{cor}
\label{cor:admissible_quadratic_duality}
	Let $\langle \kfield [\cali] ; R \rangle \cong \langle \cali; \cals , f \rangle $ be a dualizable quadratic
coalgebra, with transpose dual $\langle \kfield [\cali] ; T \rangle \cong \langle \cali; \cals ', f^! \rangle$ and
$\calj \subset \cali$ be a subset. The following conditions are equivalent:
\begin{enumerate}
	\item
the pullback quadratic coalgebra $\langle \kfield [\calj] ; R_{\kfield[\calj]} \rangle $ is $(\calj,
\cals_\calj)$-admissible;
\item
the pushforward quadratic coalgebra $\langle \kfield [\calj] ; T^{\kfield[\calj]} \rangle $ is $(\calj,
\cals'_\calj)$-admissible.
\end{enumerate}
\end{cor}

\begin{proof}
The condition of Proposition \ref{prop:admissible_pushforward} applied to the transpose dual (thus replacing $(f, \cals)
$ by $(f^! , \cals')$) is equivalent to the condition of Proposition \ref{prop:admissible_pullback}.
\end{proof}

\subsection{The weak coPBW property and surjectivity}

In general, an epimorphism of quadratic graded coalgebras is not surjective as a morphism of graded vector spaces, (see Remark \ref{rem:epimorphism_not_surjection}). In the presence of suitable bases, the surjectivity as a morphism of
graded vector spaces can be established.

\begin{defn}
	For $\cali$ a set and $\cals \subset \cali \times \cali$, let $\cals ^{(n)}\subset \cali^n$ ($n \in \nat$) be
the  subsets defined by $\cals^{(0)} = \{\emptyset \}$, $\cals^{(1)} = \cali$ and, for $n \geq 2$, $\cals^{(n)}$
recursively:
	\[
		\cals^{(n)} := (\cals^{(n-1)} \times \cali ) \cap (\cali^{\times n-2} \times \cals) .
	\]
\end{defn}

	The following definition is inspired by the notion of a PBW basis, introduced by Priddy \cite{priddy}.

\begin{defn}
	An $(\cali, \cals)$-admissible quadratic coalgebra $\langle \kfield [\cali]; R \rangle$ satisfies the weak coPBW
property if, for all $n \in \nat$, the following composite is an isomorphism
	\[
		\xymatrix{
\langle \kfield [\cali]; R \rangle_n \ar@{^(->}[r]
\ar[d]_\cong
&
\kfield [\cali]^{\otimes n}
\ar[d]^\cong
\\
\kfield [\cals^{(n)}]
&
\kfield[\cali^{\times n}].
\ar@{->>}[l]
}
	\]
(The condition is satisfied  for $n \leq 2$ by the admissibility hypothesis.)
\end{defn}

The following is clear:

\begin{lem}
	For  $(\cali, \cals)$,  $\calj$ as above and $n \in \nat$,
	$
		\cals_\calj^{(n)}
		\cong
		\cals^{(n)} \cap \calj^{\times n}.
	$
\end{lem}

\begin{prop}
\label{prop:copbw_surjectivity}
	Let $\langle \kfield [\cali]; R\rangle$ be an $(\cali, \cals)$-admissible quadratic coalgebra and $\calj \subset
\cali$ be a subset for which $\langle \kfield [\calj]; R^{\kfield [\calj]} \rangle$ is $(\calj,
\cals_\calj)$-admissible.

If the coalgebras $\langle \kfield [\cali]; R\rangle$ and $\langle \kfield [\calj]; R^{\kfield
[\calj]} \rangle $ both satisfy the weak coPBW property, then the induced morphism of quadratic coalgebras
	\[
		C: = \langle \kfield [\cali ]; R \rangle
		\rightarrow
		C^\calj : = \langle \kfield [\calj]; R^{\kfield [\calj]} \rangle
	\]
is surjective as a morphism of $\nat$-graded $\kfield$-vector spaces.
\end{prop}

\begin{proof}
For $n \in \nat$,
	the hypotheses ensure that there is a commutative diagram
	\[
		\xymatrix{
		C_n
		\ar@{^(->}[r]
		\ar[d]
		\ar@/^1pc/[rr]^\cong
		&
\kfield[\cali ^{\times n} ]
\ar@{->>}[r]
\ar@{->>}[d]
&
\kfield [\cals^{(n)}]
\ar@{->>}[d]
\\
C^\calj_n
\ar@{^(->}[r]
\ar@/_1pc/[rr]_\cong
&
\kfield[\calj ^{\times n} ]
\ar@{->>}[r]
&
\kfield [\cals_\calj^{(n)}].
		}
	\]
(The commutativity of the right hand square is a consequence of naturality of the projection and the commutativity of
the left hand square follows from the construction of the morphism.) The result follows.
\end{proof}

\section{Quadratic self-duality}
\label{sect:coalg_dual}

The concept of quadratic self-duality is important; for instance, it leads to a bijection between the set of pullback
quadratic coalgebras and the set of pushforward quadratic coalgebras of a quadratically self-dual coalgebra. Two flavours of self-duality
are considered here: the standard approach via vector-space duality and an approach in the admissible setting which allows the
finiteness hypotheses to be relaxed slightly.

\subsection{Self-duality via vector space duality}

To begin the discussion of quadratic self-duality, suppose that $\langle V; R \rangle $ is a quadratic coalgebra for
which Hypothesis \ref{hyp:ft} holds.

\begin{defn}
	The quadratic coalgebra $C:= \langle V; R \rangle $ is quadratically self-dual if there exists an isomorphism
$\phi : V \stackrel{\cong}{\rightarrow} V^*$ which induces an isomorphism of quadratic coalgebras
\[
	C= \langle V; R \rangle \stackrel{\cong} {\rightarrow} C^! = \langle V^* ; R^\perp \rangle.
\]
This condition is equivalent to the assertion that $\phi$ induces a commutative diagram
\[
	\xymatrix{
R
\ar@{^(->}[r]
\ar[d]_\cong
&
V^{\otimes 2}
\ar[d]^{\phi^{\otimes 2}}
\\
R^\perp
\ar@{^(->}[r]
&
(V^*)^{\otimes 2}.
}
\]
\end{defn}

\begin{prop}
\
\begin{enumerate}
	\item
	Let $\langle V; R\rangle $ be a quadratically self-dual coalgebra, where $\dim V$ is finite. Then
$2 \dim R = (\dim V) ^2$; in particular, $\dim V$ is even.
\item
A quadratic coalgebra $\langle V ; R \rangle $ which satisfies Hypothesis \ref{hyp:ft} is quadratically self-dual if and
only if $\langle V ; R \rangle^!$ is quadratically self-dual.
\end{enumerate}
\end{prop}

\begin{proof}
	Straightforward.
\end{proof}

\begin{exam}
	Let $\kfield$ be a field and $\cali$ be the set $\{ x, y \}$.
\begin{enumerate}
	\item
Consider the quadratic coalgebra $\langle \kfield [\cali]; R \rangle$, where $R$ is generated  by $[x]\otimes [y]$ and
$[y] \otimes [x]$. Writing $\eta_x, \eta_y$ for the dual basis of $\kfield [\cali]^* \cong \kfield^\cali$, $R^\perp$ is
generated by $\eta_x \otimes \eta_x, \eta_y \otimes \eta_y$. This quadratic coalgebra is not quadratically self-dual.
\item
The quadratic coalgebra $\langle \kfield [\cali]; T \rangle$, where $T$ is generated  by $[x] \otimes [x]$ and $[x ]
\otimes [y]$ has quadratic dual $\langle \kfield^\cali; T^\perp \rangle$, where $T^\perp $ is generated by $\eta_y
\otimes \eta_x$ and $\eta_y \otimes \eta_y$. The isomorphism $\phi$ defined by $x \mapsto \eta_y $ and $y \mapsto
\eta_x$
shows that $\langle \kfield [\cali]; T \rangle$ is quadratically self-dual.
\item
In \cite{mos}, the authors consider a $\cx$-algebra $A$ defined as the path algebra of a quiver; the algebra $A$ is quadratically self-dual \cite[page
1132]{mos}, hence dualizing gives a quadratically self-dual coalgebra.
\end{enumerate}
\end{exam}

\begin{prop}
	Let $C:= \langle V; R \rangle$ be a quadratic coalgebra which satisfies Hypothesis \ref{hyp:ft} and which is
quadratically self-dual. Then the dual quadratic algebra $C^*$ is isomorphic to the quadratic algebra $\{ V ; R \}$.
\end{prop}

\begin{proof}
	Under the finiteness hypothesis, Proposition \ref{prop:duality_quadratic} implies that $C^!$ is dual to the
quadratic algebra $\{V;R\}$. The quadratic self-duality yields an isomorphism of quadratic coalgebras $C \cong C^!$,
which gives the result.
\end{proof}

\subsection{Strict self-duality}

In the case of dualizable admissible quadratic coalgebras, there is a strict version of quadratic self-duality, restricting the class of isomorphisms considered.

\begin{defn}
	Let $\cali$, $\cals \subset \cali^{\times 2}$ be sets and $C := \langle \cali; \cals , f \rangle $ be an
$(\cali, \cals)$-admissible quadratic coalgebra which is dualizable.

The coalgebra $C$ is strictly self-dual if there exists a bijection $\sigma : \cali \stackrel{\cong}{\rightarrow} \cali$
of sets such that
\begin{enumerate}
	\item
$\sigma$ restricts to a bijection $\sigma : \cals \stackrel{\cong}{\rightarrow} \cals '$;
\item
the function $f: \cals \times \cals' \rightarrow \kfield $ satisfies the identity
$
	f = f ^! \circ (\sigma \times \sigma).
$
\end{enumerate}
\end{defn}

\begin{prop}
Suppose that $\langle \cali; \cals ,f \rangle$ is dualizable and strictly self-dual with respect to $\sigma : \cali
\stackrel{\cong}{\rightarrow} \cali$, then $\kfield [\sigma] : \kfield [\cali]
\stackrel{\cong}{\rightarrow} \kfield [\cali]$ induces an isomorphism of quadratic coalgebras $\langle \cali; \cals , f
\rangle
\stackrel{\cong}{\rightarrow} \langle \cali; \cals' , f^! \rangle$.
\end{prop}

\begin{proof}
 Clear.
\end{proof}

\subsection{Quadratic duality, pushforward and pullback}
\label{subsect:quadratic_push_pull}

Suppose that $\langle V; R \rangle $ satisfies Hypothesis \ref{hyp:ft}. Consider a subspace $W \leq V$, which defines
the restriction $R_W$ and the commutative diagram
\[
	\xymatrix{
R_W
\ar@{^(->}[r]
\ar@{^(->}[d]
&
W^{\otimes 2}
\ar@{->>}[r]
\ar@{^(->}[d]
&
W^{\otimes 2}/R_W
\ar@{^(->}[d]
\\
R
\ar@{^(->}[r]
&
V^{\otimes 2}
\ar@{->>}[r]
&
V^{\otimes 2}/R,
}
\]
in which the right hand vertical morphism is a monomorphism, since the left hand square is a pullback.

Vector space duality yields the diagram
\[
	\xymatrix{
R^\perp
\ar@{^(->}[r]
\ar@{->>}[d]
&
(V^*)^{\otimes 2}
\ar@{->>}[r]
\ar@{->>}[d]
&
R^*
\ar@{->>}[d]
\\
R^\perp_W
\ar@{^(->}[r]
&
(W^*)^{\otimes 2}
\ar@{->>}[r]
&
R_W^*
}
\]
and hence a canonical morphism of quadratic coalgebras
\[
	\langle V; R \rangle ^! = \langle V^* ; R^\perp \rangle
\rightarrow
\langle W^* ; R_W^\perp \rangle = \langle W; R_W \rangle ^!.
\]

\begin{prop}
\label{prop:self-dual-reciprocity}
	Let $\langle V; R \rangle $ be a quadratic coalgebra which satisfies Hypothesis \ref{hyp:ft} and which is
quadratically self-dual.
\begin{enumerate}
	\item
For $W \leq V$, the quadratic duality $\phi : V \stackrel{\cong}{\rightarrow} V^*$ induces  a canonical surjection of
quadratic coalgebras
\[
	\langle V; R\rangle
\rightarrow
\langle W; R_W \rangle ^!
\]
with underlying morphism $V\stackrel{\phi}{\rightarrow} V^* \rightarrow W^*$.
\item
For $V \twoheadrightarrow Z$, the quadratic duality $\phi$ induces  a canonical monomorphism of quadratic coalgebras
\[
	\langle Z; R^Z \rangle ^!
\hookrightarrow
\langle V; R \rangle
\]
with underlying morphism $Z ^* \rightarrow V^* \stackrel{\phi^{-1}}{\rightarrow} V$.
\item
These constructions are mutually inverse and define a bijection between the set of  pullback quadratic coalgebras of $\langle V;
R \rangle$ and the set of pushforward quadratic  coalgebras of $\langle V; R \rangle$.
\end{enumerate}
\end{prop}

\begin{proof}
	The construction of the associated quotient is outlined above and the second
construction is formally dual.
\end{proof}

In the case of admissible quadratic coalgebras, the above result has an analogue.

\begin{prop}
\label{prop:admissible-reciprocity}
	Let $\langle \kfield [\cali]; R \rangle \cong \langle \cali; \cals ,f  \rangle$ be a dualizable $(\cali,
\cals)$-admissible quadratic coalgebra, which is strictly self-dual with respect to the bijection $\sigma : \cali
\stackrel{\cong}{\rightarrow} \cali$. The following conditions on a set $\calj \subset \cali$ are equivalent:
\begin{enumerate}
	\item
the pullback quadratic coalgebra $\langle \kfield [\cali]; R_{\kfield [\calj]} \rangle $ is $(\calj,\cals_\calj )
$-admissible;
\item
the pushforward quadratic coalgebra $\langle \kfield [\cali]; R^{\kfield [\sigma (\calj)]} \rangle $ is $(\calj,
\cals_{\sigma (\calj)})$ admissible.
\end{enumerate}
\end{prop}

\begin{proof}
The result follows directly  from Corollary \ref{cor:admissible_quadratic_duality}; it is nevertheless  worthwhile to make the key step of the argument explicit.

By strict duality,  $f (s, s') = f^! (\sigma (s), \sigma(s'))  = f (\sigma (s'), \sigma(s))$, hence the following two conditions are equivalent:
\begin{itemize}
	\item
$f(s, s') = 0$ for all $s \in \cals_\calj$ and $s' \in \cals' \backslash \cals'_\calj$;
\item
$f(s, s') = 0$ for all $s  \in \cals \backslash \cals_{\sigma (\calj)}$ and $s' \in \cals'_{\sigma (\calj)}$,
\end{itemize}
since $\sigma$ induces a bijection between $\cals$ and $\cals'$.

The result  follows by applying Propositions \ref{prop:admissible_pullback} and \ref{prop:admissible_pushforward}.
\end{proof}

\part{Applications related to the Steenrod algebra}
\section{Bialgebras associated to additive polynomials over $\field$}
\label{sect:bialgebras}

\subsection{Graded bialgebras}

The following notion corresponds to that of a coalgebra with products introduced by Singer in \cite[Definition 2.1]{singer_hopf_ops}, although a simpler terminology
is preferred here, working over $\field$.

\begin{defn}
	The category of $\nat$-graded bialgebras over $\field$ is the category of $\nat$-graded counital comonoids in the symmetric monoidal category of commutative graded $\field$-algebras.
\end{defn}

\begin{defn}
	An $\nat$-graded bialgebra $B$ is quadratic if the underlying $\nat$-graded coalgebra is quadratic.
\end{defn}

\subsection{Monoids associated to additive polynomials over $\field$}

The terminology graded monoid is used to denote an $\nat$-graded monoid, so that the monoid is given by sets $\{ M_n | n \in \nat \}$
and the product has the form $M_i \times M_j \rightarrow M_{i+j}$. All monoids considered will be unital (with unit in $M_0$). Here, a graded monoid functor is a functor from commutative $\field$-algebras to graded monoids.

\begin{defn}\
\begin{enumerate}
	\item
	Let $\monoid $ denote the graded monoid functor defined on a commutative $\field$-algebra $A$ by
\[
\monoid_n (A) :=
	\{
\sum _{i=0} ^n \alpha_i x^{2^i} | \alpha_0=1 ,  \alpha_n \in A ^\times
\},
\]
with product  given by the composition of polynomials $(f, g ) \mapsto  f\circ g$ and  unit $x \in \monoid_0 (A)$.
\item
Let $\monoid^+$ denote the graded monoid functor defined by
\[
\monoid_n ^+(A) :=
	\{
\sum _{i=0} ^n \alpha_i x^{2^i} | \alpha_0=1 \},
\]
with product given by composition,
equipped with the monomorphism of monoid functors: $\monoid \hookrightarrow \monoid^+$ which forgets the invertibility
condition.
\end{enumerate}
\end{defn}

\begin{prop}
\label{prop:rep-graded-monoids}
\
\begin{enumerate}
	\item
The  graded monoid functor $\monoid$ is represented by the connected graded bialgebra $\bialg$, where
\[
	\bialg_n = \field [\alpha_1, \ldots , \alpha_{n-1}, \alpha_n^{\pm 1} ]
\]
and $\monoid_n (\bialg_n)$ contains the universal element $\sum _{i=0} ^n \alpha_i x^{2^i}$.
\item
For each $n \in \nat$, the commutative algebra $\bialg_n$ is graded, where $|\alpha_i| = 1- 2^i$, so that the expression
$\Sigma_i \alpha_i x^{2^i}$ is homogeneous of degree $1$, when $x$ is given degree $1$.
\item
The inclusion of graded monoid functors $\monoid \hookrightarrow \monoid^+$ is represented by a sub graded bialgebra
$\bialg^+\hookrightarrow \bialg$, where $\bialg_n^+ \cong \field[\alpha_1, \ldots, \alpha_n]$.
\item
For each $n \in \nat$, the algebra $\bialg^+_n$ is a graded coconnected algebra (concentrated in negative degrees) and
$\bialg_n \cong \bialg^+_n [\alpha_n^{-1}]$.
\end{enumerate}
\end{prop}

\begin{proof}
	Straightforward. (A key point is that the image of $\alpha_n \in \bialg_n$ under a diagonal $\Delta_{i, j}$ is an
invertible element
	of $\bialg_i \otimes \bialg_j$.)
\end{proof}

\begin{rem}
\
\begin{enumerate}
\item
	To fix the conventions used in defining the bialgebra structure, consider $\Delta_{1,1} : \bialg_2 \rightarrow
\bialg_1 \otimes \bialg_1$, which is the morphism of $\field$-algebras $\field [ \alpha_1 , \alpha_2 ^{\pm 1}]$ which is
defined by
\begin{eqnarray*}
\alpha_1 & \mapsto & 1 \otimes \alpha_1 + \alpha _1 \otimes 1 \\
\alpha_2 & \mapsto & \alpha_1 ^2 \otimes \alpha_1 .
\end{eqnarray*}

\item
	The bialgebra $\bialg^+$  determines the bialgebra $\bialg$ by the localization which corresponds to inverting
the element
	$\alpha_n$ of $\bialg_n$, for each $n$.
\item
For each $n$, the algebra $\bialg^+_n$ is of finite type; for $n\geq 2$, this is not true of $\bialg_n$.
\end{enumerate}
\end{rem}

\subsection{Group completion}

There is a familiar ungraded monoid functor, which can be interpreted as having underlying functor $\monoid_\infty^+$, by passing to  formal power series:

\begin{defn}
Let $\monoid_\infty^+$ be the monoid functor defined by
\[
	\monoid_\infty^+ (A )= \{ \sum _{i=0} ^\infty \alpha_i x^{2^i} | \alpha_0=1 \},
\]
with product given by the composition of formal power series.
\end{defn}

\begin{rem}
	This is  a group functor, since composition inverses can be constructed working with formal power
series.
\end{rem}

The following is well-known:

\begin{prop}
	The functor $\monoid_\infty^+$ is represented by the Hopf algebra $\dualst \cong \field [\xi_i | i \geq 0, \xi_0
=1] $, the dual Steenrod algebra.
\end{prop}

\begin{nota}
	Let $\catmonoid$ denote the category of monoids (with unit) and $\grcatmon$ the category of $\nat$-graded
monoids (with unit).
\end{nota}

There is a forgetful functor $\calo :\grcatmon \rightarrow \catmonoid$, which sends a graded monoid $(M_n| n \in \nat)$
to the monoid $\amalg_n M_n$.

\begin{prop}
	The functor $\calo :\grcatmon \rightarrow \catmonoid$ admits a right adjoint, $\gamma : \catmonoid \rightarrow
\grcatmon$, which associates to a monoid $N$ the graded monoid $\gamma N$ with $(\gamma N)_s = N$ for all $s$ and
structure morphisms induced by the product of $N$.
\end{prop}

\begin{proof}
	Straightforward.
\end{proof}

This allows the following connection to be made between the dual Steenrod algebra and the graded bialgebra $\bialg^+$.

\begin{prop}
	There is a morphism of monoid functors
\[
	\calo (\monoid^+) \rightarrow \monoid^+_\infty,
\]
induced by considering a polynomial as a formal power series.

The adjoint morphism $
	\monoid^+ \rightarrow \gamma \monoid^+_\infty
$
is induced by the morphism of graded bialgebras
 $
	\gamma \dualst \rightarrow \bialg^+
 $
(where $\gamma$ is the analogue for bialgebras of the functor defined above), which has components
\[
	\theta _n : \field [\xi_i ] \rightarrow \bialg_n^+ \cong \field [\alpha_1, \ldots , \alpha_n]
\]
$\xi_i \mapsto \alpha _i$, for $1 \leq i \leq n$, and $\xi_i \mapsto 0$, for $i>n$.
\end{prop}

\begin{proof}
	Clear.
\end{proof}

\begin{rem}
	The morphism $\calo (\monoid^+) \rightarrow \monoid^+_\infty$ can be viewed as
 a group completion of the monoid $\calo (\monoid^+)$.
\end{rem}

\subsection{Change of generators for $\bialg$}

There is a second standard choice of generators for the algebra $\bialg_n$ (for each $n \in \nat$). In the following, by
convention, $\alpha_0 =1$.

\begin{nota}
	Let $Q_{n,i}$ denote the element $\frac{\alpha_i}{\alpha_n}$ (for $0\leq i \leq n$) of $\bialg_n$. (For other
integers $i$, $Q_{n,i}$ is taken to be zero.) Hence
$|Q_{n,i}| = 2^n - 2^i$.
\end{nota}

\begin{lem}
\label{lem:Delta_1,1_Q}
	For $n \in \nat$, there is an isomorphism of algebras:
\[
	\bialg_n \cong \field [Q_{n,0}^{\pm 1}, Q_{n,1}, \ldots, Q_{n,n-1}].
\]
	The diagonal $\Delta_{1,1}: \bialg_2 \rightarrow \bialg _1 \otimes \bialg_1 $ is the morphism of
$\field$-algebras determined by
\begin{eqnarray*}
	Q_{2,0} & \mapsto & Q_{1,0}^2 \otimes Q_{1,0}
\\
Q_{2,1} & \mapsto & Q_{1,0} \otimes Q_{1,0} + Q_{1,0}^2 \otimes 1.
\end{eqnarray*}
\end{lem}

\begin{proof}
	Straightforward.
\end{proof}

\begin{rem}
In \cite{singer_invt_lambda}, Singer uses invariant theory to construct a graded coalgebra $\Gamma$, where $\Gamma_n
\cong \field [Q_{n,0}^{\pm 1} , \ldots, Q_{n,n-1}]$. The coproducts of the generators correspond to the coproducts of
the Dyer-Lashof algebra (Cf. \cite[(2.14)]{singer_invt_lambda}).
\end{rem}

\begin{prop}
\label{prop:Gamma_iso_bialg}
The graded coalgebra $\Gamma$ has the structure of a graded bialgebra.
\end{prop}

\begin{proof}
A generalization of  Lemma \ref{lem:Delta_1,1_Q}.
\end{proof}

\begin{rem}
The graded bialgebra $\Gamma$ contains the sub-bialgebra $\Gamma^-$, which is defined by
$
	\Gamma^-_n := \field [Q_{n,0} , \ldots, Q_{n,n-1}];
$
this defines a sub-bialgebra $\bialg^-$ of $\bialg$.
\end{rem}

\subsection{The coalgebras $\bialg$, $\bialg^+$ are quadratic}

In this section it is shown that the underlying coalgebras of $\bialg$ and $\bialg^+$ are quadratic. (This result can
also  be deduced from the results of \cite[Section 2]{singer_hopf_ops}.)

\begin{prop}
\label{prop:cogeneration}
	The underlying graded coalgebras  of $\bialg^+ $ and of $\bialg$ are  cogenerated in degree one.
\end{prop}

\begin{proof}
A straightforward localization argument shows that it is sufficient to show that $\bialg^+$ is cogenerated in degree
one.
Moreover,  it suffices to prove that, for each $n \geq 2$, the diagonal $\Delta_{n-1, 1} : \bialg_n^+ \rightarrow
\bialg_{n-1}^+ \otimes \bialg_1^+ $ is a monomorphism.

The diagonal $\Delta_{n-1, 1} : \field [\alpha_1, \ldots , \alpha_n] \rightarrow
\field[\alpha_1 , \ldots , \alpha_{n-1}] \otimes \field[\alpha_1]$ is determined by
\begin{eqnarray*}
	\alpha_1 & \mapsto & 1 \otimes \alpha_1 + \alpha_1 \otimes 1 \\
\alpha_s & \mapsto & \alpha_{s-1}^2 \otimes \alpha_1 + \alpha_s \otimes 1 \\
\alpha_n & \mapsto  & \alpha_{n-1}^2 \otimes \alpha_1,
\end{eqnarray*}
where $1\leq  s < n$.

We require to prove that the images $\alpha_i'$ of the elements $\alpha_i$ are algebraically independent. By composing
with the projection $p : \field[\alpha_1, \ldots , \alpha_{n-1}] \otimes \field [\alpha_1] \twoheadrightarrow
\field[\alpha_1, \ldots , \alpha_{n-1}]$ induced by  $\field [\alpha_1] \rightarrow \field$, $\alpha_1 \mapsto 0$, it is
straightforward to show that the elements $\alpha_1' ,\ldots , \alpha_{n-1} '$ are algebraically independent.

Suppose that there exists a nonzero polynomial $f \in \field [\alpha_1', \ldots , \alpha_{n-1}'] [x] $ such that
$f(\alpha_n' )=0$, considered as an element of $\field[\alpha_1, \ldots , \alpha_{n-1}] \otimes \field [\alpha_1]$ and
choose such an $f$ of minimal degree (in $x$). Since $p (\alpha'_n)=0$, one has $p (f(\alpha_n'))= p (f_0)=0$. It
follows that $f_0 =0$, hence $f(x)$ can be replaced by $g(x):= f(x)/x$, which is again non-trivial and which satisfies
$g(\alpha'_n)=0$. This contradicts the hypothesis that the degree of $f$ is minimal.
\end{proof}

\begin{nota}
 Let $n$ be a positive integer and $1 \leq a \leq n$ be an integer. Write
 \begin{enumerate}
  \item
  $j_n : \bialg_{n-1}^+ \hookrightarrow \bialg_n^+$ for the canonical inclusion of algebras
  $\field [\alpha_1 , \ldots , \alpha_{n-1}] \hookrightarrow \field [\alpha_1 , \ldots , \alpha_{n}]$;
  \item
$p_a : (\bialg^+_1 ) ^{\otimes n}\twoheadrightarrow  (\bialg^+_1 ) ^{\otimes n-1}$
   for the projection of algebras induced by the augmentation of the $a^{\mathrm{th}}$ factor;
   \item
   $\delta_a : \bialg_n^+ \rightarrow (\bialg^+_1 ) ^{\otimes n-1}$ for the composite of the iterated coproduct with
$p_a$.
 \end{enumerate}
\end{nota}

\begin{lem}
\label{lem:alpha_divisibility}
 Let $n$ be a positive integer and $1 \leq a \leq n$ be an integer.
 \begin{enumerate}
  \item
  The composite
  $\bialg_{n-1}^+ \stackrel{j_n} {\rightarrow } \bialg_n^+ \stackrel{\delta_a} {\rightarrow } (\bialg^+_1 ) ^{\otimes
n-1}$ is the iterated coproduct.
\item
The kernel of $\delta_a : \bialg_n^+ \rightarrow (\bialg^+_1 ) ^{\otimes
n-1}$ is the ideal $\alpha_n \bialg_n^+$.
 \end{enumerate}
 \end{lem}

\begin{proof}
The first statement follows by considering the natural transformation which is represented by the composite morphism. In
particular, this
identification shows that the composite is a monomorphism. The second statement follows, by considering the decomposition as $\bialg_{n-1}^+$-modules $\bialg_n^+ \cong \bialg_{n-1}^+ \oplus \alpha_n \bialg_n^+$ given by using the inclusion
$j_n$.
\end{proof}

To avoid notational confusion, write  $\field[u_1, \ldots , u_n]$ for the algebra $(\bialg^+_1) ^{\otimes n}$, where $u_i$ corresponds to the generator
$\alpha_1$ of the $i$th tensor factor; write $\alpha_i'$ for the image in $(\bialg^+_1) ^{\otimes n}$ of the generator $\alpha_i$.

\begin{thm}
\label{thm:bialg_quad}
	The underlying coalgebras of $\bialg$ and of $\bialg^+$ are quadratic.
\end{thm}

\begin{proof}
	By a  localization argument (clearing fractions), it is sufficent to show that $\bialg^+$ is quadratic.
Moreover,
 by induction upon $n$, it is sufficient to show that, for $n \geq 3$, the intersection of $\bialg^+_{n-1} \otimes
\bialg^+_1$ and $\bialg^+_1 \otimes \bialg^+_{n-1}$ in $(\bialg^+_1) ^{\otimes n}$ is equal to $\bialg^+_n$.

The diagonals $\Delta_{n-1,1}$ and $\Delta_{1,n-1}$ can be calculated explicity; for notational clarity, different sets
of  generators $\{\beta_i
\}$ and $\{ \gamma_i \}$ for $\bialg^+_{n-1}$ are used. By definition, the morphism $\Delta _{n-1,1} : \bialg^+_n \rightarrow \bialg^+_{n-1}
\otimes \bialg^+_1 \cong \field [\beta_1 , \ldots , \beta_{n-1}] \otimes \field [u_n]$ is determined by the coefficients
of the composition
\[
	( x+u_n x^2) \circ (x + \beta_1 x^2 + \beta_2 x^4 + \ldots + \beta_{n-1} x^{2^{n-1}} ),
\]
so $\Delta _{n-1,1}$ is given by:
\begin{eqnarray*}
	\alpha_1 &\mapsto & 1 \otimes u_n + \beta_1 \otimes 1\\
\alpha_s & \mapsto & \beta_{s-1}^2 \otimes u_n + \beta_{s} \otimes 1 \ \  \ \ (1 < s < n) \\
\alpha_n & \mapsto & \beta_{n-1}^2 \otimes u_n.
\end{eqnarray*}

Similarly, $\Delta _{1,n-1} : \bialg^+_n \rightarrow \bialg^+_1 \otimes  \bialg^+_{n-1}  \cong  \field [u_1]\otimes
\field [\gamma_1 , \ldots , \gamma_{n-1}] $ is determined by the coefficients of the composition
\[
	 (x + \gamma_1 x^2 + \gamma_2 x^4 + \ldots + \gamma_{n-1} x^{2^{n-1}} )\circ (x+u_1 x^2),
\]
so that $\Delta _{1,n-1}$ is given by:
\begin{eqnarray*}
	\alpha_1 &\mapsto &  u_1  \otimes 1 + 1 \otimes \gamma_1 \\
\alpha_s & \mapsto & u_1^{2^{s-1}}\otimes \gamma_{s-1} + 1 \otimes \gamma_{s}  \ \ \ \  (1 < s < n) \\
\alpha_n & \mapsto &  u_1^{2^{n-1}} \otimes \gamma_{n-1}.
\end{eqnarray*}

Hence there are isomorphisms of subalgebras of $\field [u_1, \ldots , u_n]$:
\begin{eqnarray*}
	\bialg^+_{n-1} \otimes \bialg^+_1 &\cong & \field [\alpha_1', \ldots , \alpha'_{n-1}, u_n] \\
\bialg^+_1 \otimes \bialg^+_{n-1}  &\cong & \field [\alpha_1', \ldots , \alpha'_{n-1}, u_1].
\end{eqnarray*}

Let  $X$ belong to the intersection; considering $X$ in $\field [\alpha_1', \ldots ,
\alpha'_{n-1}, u_1]$, by subtracting an appropriate element of $\field [\alpha_1', \ldots , \alpha'_{n-1}]$, we may suppose that $X$
is divisible by $u_1$. If $X =0$, there is nothing to prove; otherwise, by an inductive argument based on the degree, it is
sufficient to show that $X$ is divisible by $\alpha'_n$, using the fact that,  if $X = Y \alpha'_n$, then $Y$
belongs to the intersection (this follows from the identifications $\alpha_n' = \beta_{n-1}^2 u_n = u_1^{2^{n-1} }\gamma_{n-1}$).

Considering $X$ as an element of  $\bialg^+_{n-1} \otimes \bialg^+_1$, $X$ belongs to the kernel of $\delta_1 \otimes \bialg^+_1$, hence by Lemma \ref{lem:alpha_divisibility},  $X$ is divisible by $\beta_{n-1}= u_1^{2^{n-2}}\ldots u_{n-1}$, in particular, is divisible by $u_{n-1}$.

Using $u_{n-1}$-divisibility and considering $X$ as an element of  $\bialg^+_1 \otimes \bialg^+_{n-1}$, Lemma \ref{lem:alpha_divisibility} (applied with respect to $\bialg^+_1 \otimes \partial_{n-2}$) implies that $X$ is divisible by $\gamma_{n-1} = u_2^{2^{n-2}}\ldots u_{n-1}^2 u_{n}$, so that the element $ X' := X/ \beta_{n-1}$
 of $ \bialg^+_{n-1} \otimes \bialg^+_1$ is divisible by $u_{n-1}u_n$.

Repeating this argument for $X' \in \bialg^+_{n-1} \otimes \bialg^+_1$, which is $u_{n-1}$-divisible, it follows that $X'$ is divisible by $\beta_{n-1}$ and by $u_n$. Hence, $X$ is divisible by $\alpha'_n = \beta_{n-1}^2 u_n$, as required.
\end{proof}

\section{Admissibility and duality for $\bialg$}
\label{sect:duality}

The main result, Theorem \ref{thm:bialg_selfdual}, of this section shows that the underlying graded coalgebra of $\bialg$ is strictly self-dual.

\subsection{Admissibility for $\bialg$}

The graded bialgebra $\bialg$ has an internal degree;  the underlying graded vector space of $\bialg_1$ is isomorphic to $\field [\zed]$ and the subspace $\bialg^+_1$ is
isomorphic to $\field [\nat]$ as graded vector spaces.
(Here the  {\em cohomological} grading has been adopted.)

\begin{defn}
	Let
\begin{enumerate}
	\item
$\cals \subset \zed \times \zed$ denote the subset $\{(i, j)   | i \geq 2j\}$;
\item
let $\cals'$ denote the complement $\cals':= \{(i, j)   | i < 2j\}$.
\end{enumerate}
\end{defn}

\begin{nota}
	Let $\partial: \zed \times \zed \rightarrow \zed \times \zed$ denote the bijection $(i, j) \mapsto (i+2, j+1)$,
which restricts to bijections
\begin{eqnarray*}
	\partial : \cals &\stackrel{\cong}{\rightarrow} & \cals\\
\partial : \cals' &\stackrel{\cong}{\rightarrow} & \cals'.
\end{eqnarray*}
\end{nota}

\begin{lem}
	Let $(i,j)$ be an element of $\cals$ then $(i, j) = \partial^j (e, 0)$, where $e= i-2j\geq 0$ is the excess, so
that $(e,0) \in \cals$.
\end{lem}

The following provides an ordering result which is useful for constructing admissible bases (compare \cite[Chapter 4,
Section 1]{polis_posit}).

\begin{lem}
\label{lem:admissible_order}
	Let $(i, j) \in \cals$ and $(l,m) \in \cals'$ such that $i+j = l+m$, then $l < i$.
\end{lem}

\begin{prop}
\label{prop:bialg_admissible}
	The graded coalgebra $\bialg$ is $(\cali, \cals)$-admissible.
\end{prop}

\begin{proof}
	Recall that $\bialg_2\hookrightarrow \bialg_1 \otimes \bialg_1 \cong \field[x^{\pm 1}, y^{\pm 1}] $ is induced
by localization of the algebra morphism $\field [\alpha_1, \alpha_2] \rightarrow \field [x,y]$ given by
\begin{eqnarray*}
\alpha_1 & \mapsto & x+y \\
\alpha_2 & \mapsto & x^2 y .
\end{eqnarray*}
To prove the result, it suffices to show that, for each $(i, j) \in \cals$, there exists a unique element $h_{i,j}$ of
$\field [\alpha_1 , \alpha_2^{\pm 1}] $ such that $h_{i,j}$ maps to $x^iy^j$ modulo $\field [\cals']$. Since $\alpha_2
^j h_{(e,0)}$ satisfies this condition, where $e= i - 2j$, it suffices to construct the elements $h_{(e,0)}$. The
existence and the unicity of these elements can be shown by standard methods, based on Lemma \ref{lem:admissible_order}
and the left lexicographical order on $\zed \times \zed$. (An explicit construction  of the elements $h_{e,0}$ is given by the following result.)
\end{proof}

\begin{prop}
\label{prop:coefficients_h(e,0)}
	The elements $h_{(e,0)}$, for $e \in \nat$, are defined recursively by
\begin{eqnarray*}
	h_{(0,0)} &=& 1\\
	h_{(e+1, 0) }& = &\alpha_1 h_{(e,0)} + \alpha_2 h_{(e-2, 0)},
\end{eqnarray*}
where $h_{(k,0)}$ is taken to be zero if $k <0$.

Moreover, the image of $h_{(e,0)}$ in $\field [x,y]$ is
\[
	x^e + \sum _u \binom{u-1}{2u-e -1} x^{e-u} y^u.
\]
\end{prop}

\begin{proof}
	The proof is by induction; for $e \leq 2$, the result is immediate, hence we may suppose that $e >2$.
Observe that the binomial coefficient is zero for $2u - e -1<0$; if $2u -e -1 \geq 0$, the excess corresponding to the
monomial $x^{e-u} y^u$ is $(e-u)-2u = e-3u$, which is negative (since $e$ is non-negative and $e-2u \leq -1$). Hence the
given polynomial has the required form. Thus, to establish the result, by unicity, it suffices to show that the given
monomials $h_{(e,0)}$ have the stated images.

The image of $\alpha_1 h_{(e,0)}$ in $\field [x,y]$ is
\[
	x^{e+1} + \Big \{ \sum _u \binom{u-1}{2u-e -1} x^{e-u+1}y^u \Big\} +  x^e y + \sum _v \binom{v-2}{2v-e -3}
x^{e-v+1}y^v,
\]
where the second sum has been reindexed by $v=u+1$.

Similarly, the image of $\alpha_2 h_{(e-2, 0)}$ is
\[
	x^ e y + \sum_v \binom{v-2}{2v-e -1} x^{e-v+1} y ^v.
\]
Using the relations
\[
	\binom {u-2}{2u-e -2} + \binom{u-2}{2 u -  e -3} = \binom {u-1}{2u -e -2} = \binom {u-1}{2u - (e+1) -1}
\]
completes the inductive step.
\end{proof}

\begin{cor}
\label{cor:f_bialg}
The underlying quadratic coalgebra of $\bialg$ is isomorphic to $\langle \zed; \cals , f \rangle$, where the
function $f : \cals \times \cals' \rightarrow \field$  is given on pairs $((i,j), (l,m)) \in \cals \times \cals'$ such
that $i+j = l+m$ by
\[
	f((i,j), (l,m))
=
\binom{m-j-1}{2 (m-j)- e -1}
=
\binom{m-j-1}{m+j-l-1}
,
\]
where $e = i -2j$; on pairs such that $i+j \neq l+m$, the function is trivial.

In particular, the quadratic coalgebra $\langle \zed; \cals , f \rangle$ is dualizable.
\end{cor}

\begin{proof}
The proof of Proposition \ref{prop:bialg_admissible} indicates that there is a relation
\[
 f ((i,j), (l,m) ) = f ((i-2j, 0), (l-2j, m-j))
\]
and $i-2j$ is the excess $e$. Moreover, Proposition \ref{prop:coefficients_h(e,0)} shows that
\[
 f((e,0), (u,v)) = \binom {u-1}{2u -e -1}
\]
if $u+v= e$ and is trivial otherwise.

The first identity follows immediately (using the fact that $i+j = l+m$ is equivalent to the corresponding condition
$u+v=e$, where $u=l-2j$, $v=m-j$).

To show that $\langle \zed; \cals , f \rangle$ is dualizable, it is sufficient to show that, for a fixed $(l,m) \in
\cals'$, the set of $(i, j) \in \cals$ such that  $f((i,j),(l,m)) \neq 0$ is finite.

It suffices to consider $(i, j) \in \cals$ such that $i+j = l+m$. The binomial coefficient
\[
 \binom{m-j-1}{m+j-l-1}
\]
is zero if $m-j-1<0$ or if $m+j -l-1<0$. It follows that there is a finite interval of values of $j$ for which the
binomial coefficient can be non-trivial. This implies the result.

\end{proof}

\begin{cor}
	The quadratic coalgebra $\bialg^+$ is $(\nat, \cals_\nat)$-admissible.
\end{cor}

\begin{proof}
	An immediate consequence of Proposition \ref{prop:admissible_pullback}. (The result can be seen directly.)
\end{proof}

\subsection{The weak coPBW property}

\begin{prop}
\label{prop:weak_coPBW}
 The $(\zed, \cals)$-admissible quadratic bialgebra $\bialg$ satisfies the weak coPBW property.
\end{prop}

\begin{proof}
 (Indications.) It is straightforward to see that it is sufficient to prove the result for $\bialg^+$, which is $(\nat,
\cals_\nat)$-admissible. Here standard inductive calculations apply, generalizing the methods used in the proof that
$\bialg^+$ is quadratic and that it is $(\nat, \cals_\nat)$-admissible. (Cf. \cite[Lemma 2.7]{singer_invt_lambda}, which
 establishes an analogous result.)
\end{proof}

\begin{rem}
\begin{enumerate}
\item
 A detailed proof is not given, since this result is related to the fact that the universal Steenrod algebra has a PBW
basis. (See the next section.)
\item
 The existence of such a basis is  intimately related to the argument used by Singer in \cite[Proposition 8.1]{singer_invt_lambda} to prove that the dual of the opposite
 of the Lambda algebra is a quotient coalgebra of $\Gamma$ (compare Proposition \ref{prop:copbw_surjectivity}).
\end{enumerate}
\end{rem}

\subsection{Self-duality}

Recall that the coalgebra underlying the bialgebra $\bialg$ is  isomorphic to the quadratic coalgebra $\langle \zed;
\cals , f \rangle $ and is dualizable;  the transpose dual coalgebra is $\langle \zed; \cals', f^! \rangle$.

\begin{thm}
\label{thm:bialg_selfdual}
 The quadratic coalgebra $\langle \zed; \cals , f \rangle $ is strictly  self-dual with respect to the
bijection $\sigma : \zed \rightarrow \zed$, $n \mapsto 1 -n $.
\end{thm}

\begin{proof}
 It is elementary to check that $\sigma$ restricts to a bijection between $\cals$ and $\cals'$ (note  that
$\sigma^2$ is the identity). Hence, it suffices to show that, for $((i,j),(l,m)) \in \cals \times \cals'$, there is an
equality
 \[
  f ((i,j) , (l,m)) = f ((\sigma (l,m) , \sigma (i,j))= f ((1-l, 1-m), (1-i, 1-j)).
 \]
Clearly it is sufficient to restrict to the case $i+j = l+m$. The right hand side is given by the binomial coefficient
\[
 \binom{(1-j)-(1-m)-1}{(1-j)+(1-m) - (1-i)-1 } = \binom {m-j-1}{i-m-j}.
\]
This is equal to $f ((i,j) , (l,m))$, since $i-m-j=l-2j$ (using $i+j = l+m$) and
$m+j -l-1 = (m-j-1)-(l-2j)$.
\end{proof}

\section{Applications to the universal Steenrod algebra}
\label{sect:universal}

\subsection{The universal Steenrod algebra}

The universal Steenrod algebra was introduced by May in \cite{may_gen_alg} (under the name big Steenrod algebra) in his
general algebraic approach to
the construction of Steenrod operations; here we consider only the mod-$2$ case. The terminology universal Steenrod
algebra was introduced by Lomonaco, since the Dyer-Lashof algebra, the opposite of the Lambda algebra and the Steenrod
algebra all appear as sub-quotients of the algebra. No algebraic universal property is claimed.

The following presentation suffices here:

\begin{defn}(Cf. \cite{bcl_cohomology}, for example.)
 The universal Steenrod algebra over $\field$ is the quadratic algebra $\univsteen$ generated by elements $\{ y^i | i
\in \nat \}$, with $y_i$ of internal degree $i$, subject to the generalized Adem relations
\[
 y_{2k -1 -n} y_k
 =
 \sum_j
 \binom {n-1-j}{j} y_{2k -1 -j} y_{k+j -n}.
\]
\end{defn}

\begin{lem}
\label{lem:gen_Adem_relations}
	The generalized Adem relations are equivalent to:
\[
	y_u y_v =
\sum_m
\binom{v-m-1}{v+m-u-1} y_{u+v -m} y_m,
\]
where $u <2v$.
\end{lem}

\begin{defn}
 Let  $\atilde$ denote the quotient of $\univsteen$ by the ideal generated by $\{ y_j | j <0 \}$. Equivalently,
$\atilde$ is the quadratic algebra generated by the
elements $\{ Sq ^i | i \in \nat \}$ (with internal degree $|Sq^i | = i$) subject to the Adem relations (without the
condition $Sq^0= 1$). (This algebra is denoted $\mathbf{\mathcal{B}}$ in \cite{singer_hopf_ops}.)
\end{defn}

\subsection{Self-duality and the relationship between $\bialg$ and $\univsteen$}

Transpose duality  for admissible quadratic coalgebras has an obvious counterpart for quadratic algebras, as does  self-duality:

\begin{defn}
Let $\cali$ be a set and $\cals \subset \cali \times \cali$ be a subset. A quadratic algebra $\{ \kfield [\cali]; R
\}$  is $(\cali, \cals)$-admissible if and only if the quadratic coalgebra
$\langle \kfield [\cali]; R \rangle$ is $(\cali, \cals)$-admissible; in this case, the quadratic algebra $\{ \kfield
[\cali]; R
\}$
\begin{enumerate}
\item
is dualizable if and only if $\langle \kfield [\cali]; R \rangle$ is dualizable;
\item
is strictly self-dual if and only if $\langle \kfield [\cali]; R \rangle$ is strictly self-dual.
\end{enumerate}
\end{defn}

Recall that the bialgebra $\bialg$ is isomorphic (as a graded coalgebra)  to the admissible coalgebra
$\langle \zed; \cals , f \rangle$, in the notation of Section \ref{sect:duality}.

\begin{thm}
\label{thm:univsteen_self-dual}
\
\begin{enumerate}
\item
 The universal Steenrod algebra $\univsteen$ is isomorphic to the quadratic algebra associated to the quadratic
coalgebra $\langle \zed; \cals' , f ^! \rangle$
\item
 The universal Steenrod algebra $\univsteen$ is dualizable and is strictly self-dual.
\end{enumerate}
\end{thm}

\begin{proof}
For the first statement, compare the coefficients of the generalized Adem relations (in the form given in Lemma
\ref{lem:gen_Adem_relations}) with the function $f$.

The second statement follows from  Theorem \ref{thm:bialg_selfdual}.
\end{proof}

\begin{rem}
\begin{enumerate}
 \item
It is known that the universal Steenrod algebra can be constructed as the
quadratic algebra $\{ \field [\zed]; \Gamma_2 \}$, where $\Gamma_2$ is the degree two part of Singer's graded bialgebra
$\Gamma$ \cite{lomonaco_dickson}. This presentation of the universal Steenrod algebra \cite{lomonaco_dickson} can be recovered using
self-duality.

Quadratic self-duality can be deduced from  \cite{lomonaco_reciprocity} (where a weaker reciprocity result is established, since the considerations are motivated by the study of Koszul, finite-type objects). Moreover, the
calculation of the diagonal cohomology of the universal Steenrod algebra \cite{lomonaco_diagonal} is essentially
equivalent to this theorem, as is the main result of \cite{lomonaco_coalgebra}.
\item
 The reciprocity results of Section \ref{subsect:quadratic_push_pull},   Proposition \ref{prop:self-dual-reciprocity} and Proposition \ref{prop:admissible-reciprocity}, apply to the bialgebra $\bialg$ to  recover the results of Lomonaco \cite{lomonaco_reciprocity}.
\end{enumerate}
\end{rem}

Recall that each algebra $\bialg^+_n$ is of finite type (with respect to the internal degree), hence taking the vector space dual (that is the dual in the naïve sense) yields a quadratic algebra.

\begin{cor}
\cite[Theorem 1.2, Proposition 2.11]{singer_hopf_ops}
\label{cor:bialg}
	The quadratic algebra $\atilde$ is the  dual of the quadratic bialgebra $\bialg^+$.
\end{cor}

\begin{rem}
 Analogous results hold at odd primes; this will be developed elsewhere.
\end{rem}


\begin{thebibliography}{MOS09}

\bibitem[BCL05]{bcl_cohomology}
Maurizio Brunetti, Adriana Ciampella, and Luciano~A. Lomonaco, \emph{The
  cohomology of the universal {S}teenrod algebra}, Manuscripta Math.
  \textbf{118} (2005), no.~3, 271--282. \MR{2183040 (2006g:55022)}

\bibitem[Lom90]{lomonaco_dickson}
Luciano Lomonaco, \emph{Dickson invariants and the universal {S}teenrod
  algebra}, Rend. Circ. Mat. Palermo (2) Suppl. (1990), no.~24, 429--443,
  Fourth Conference on Topology (Italian) (Sorrento, 1988). \MR{1108226
  (92e:55017)}

\bibitem[Lom92]{lomonaco_reciprocity}
\bysame, \emph{A phenomenon of reciprocity in the universal {S}teenrod
  algebra}, Trans. Amer. Math. Soc. \textbf{330} (1992), no.~2, 813--821.
  \MR{1044963 (92f:55027)}

\bibitem[Lom97]{lomonaco_diagonal}
\bysame, \emph{The diagonal cohomology of the universal {S}teenrod algebra}, J.
  Pure Appl. Algebra \textbf{121} (1997), no.~3, 315--323. \MR{1477613
  (98f:55016)}

\bibitem[Lom06]{lomonaco_coalgebra}
\bysame, \emph{On {S}inger's algebra and coalgebra structures}, Boll. Unione
  Mat. Ital. Sez. B Artic. Ric. Mat. (8) \textbf{9} (2006), no.~3, 611--617.
  \MR{2274115 (2007g:55018)}

\bibitem[May70]{may_gen_alg}
J.~Peter May, \emph{A general algebraic approach to {S}teenrod operations}, The
  {S}teenrod {A}lgebra and its {A}pplications ({P}roc. {C}onf. to {C}elebrate
  {N}. {E}. {S}teenrod's {S}ixtieth {B}irthday, {B}attelle {M}emorial {I}nst.,
  {C}olumbus, {O}hio, 1970), Lecture Notes in Mathematics, Vol. 168, Springer,
  Berlin, 1970, pp.~153--231. \MR{0281196 (43 \#6915)}

\bibitem[MM65]{milnor_moore}
John~W. Milnor and John~C. Moore, \emph{On the structure of {H}opf algebras},
  Ann. of Math. (2) \textbf{81} (1965), 211--264. \MR{0174052 (30 \#4259)}

\bibitem[MOS09]{mos}
Volodymyr Mazorchuk, Serge Ovsienko, and Catharina Stroppel, \emph{Quadratic
  duals, {K}oszul dual functors, and applications}, Trans. Amer. Math. Soc.
  \textbf{361} (2009), no.~3, 1129--1172. \MR{2457393 (2010c:16029)}

\bibitem[PP05]{polis_posit}
Alexander Polishchuk and Leonid Positselski, \emph{Quadratic algebras},
  University Lecture Series, vol.~37, American Mathematical Society,
  Providence, RI, 2005. \MR{2177131 (2006f:16043)}

\bibitem[Pri70]{priddy}
Stewart~B. Priddy, \emph{Koszul resolutions}, Trans. Amer. Math. Soc.
  \textbf{152} (1970), 39--60. \MR{0265437 (42 \#346)}

\bibitem[PV95]{posit_vishik}
Leonid Positselski and Alexander Vishik, \emph{Koszul duality and {G}alois
  cohomology}, Math. Res. Lett. \textbf{2} (1995), no.~6, 771--781. \MR{1362968
  (97b:12008)}

\bibitem[Sin83]{singer_invt_lambda}
William~M. Singer, \emph{Invariant theory and the lambda algebra}, Trans. Amer.
  Math. Soc. \textbf{280} (1983), no.~2, 673--693. \MR{MR716844 (85e:55029)}

\bibitem[Sin05]{singer_hopf_ops}
\bysame, \emph{On the algebra of operations for {H}opf cohomology}, Bull.
  London Math. Soc. \textbf{37} (2005), no.~4, 627--635. \MR{2143743
  (2006c:55014)}

\end{thebibliography}
\providecommand{\bysame}{\leavevmode\hbox to3em{\hrulefill}\thinspace}
\providecommand{\MR}{\relax\ifhmode\unskip\space\fi MR }
\providecommand{\MRhref}[2]{%
  \href{http://www.ams.org/mathscinet-getitem?mr=#1}{#2}
}
\providecommand{\href}[2]{#2}

\end{document}